\documentclass[11pt]{amsart}
\usepackage{geometry, amssymb, dsfont, amsaddr}
\usepackage[all]{xy} 
\title{Non-commutative $q$-expansions}
\author{Mahesh Kakde}
\date{version 2}
\address{King's College London}
\email{mahesh.kakde@kcl.ac.uk}

\newtheorem{theorem}{Theorem}

\newtheorem{proposition}[theorem]{Proposition}
\newtheorem{remark}[theorem]{Remark}

\newtheorem{definition}[theorem]{Definition}
\newtheorem{corollary}[theorem]{Corollary}

\begin{document}

\maketitle

\begin{abstract} In this short note we partially answer a question of Fukaya and Kato by constructing a $q$-expansion with coefficients in a non-commutative Iwasawa algebra whose constant term is a non-commutative $p$-adic zeta function. 

\end{abstract}

\tableofcontents

\section*{Notations and Set up} We use the following notation and set up throughout the paper. Fix an odd prime $p$. For a pro-finite group $G$ we define the Iwasawa algebra $\Lambda(G):= \varprojlim_U \mathbb{Z}_p[G/U]$, where $U$ runs through open normal subgroups of $G$. If $G$ is a compact $p$-adic Lie group with a closed normal subgroup $H$ such that $G/H \cong \mathbb{Z}_p$, the additive group of $p$-adic integers then we have the canonical Ore set of \cite{CFKSV:2005} defined as 
\[
S := \{f \in \Lambda(G) : \Lambda(G)/\Lambda(G)f \text{ is a f.g. } \Lambda(H)-\text{module}\}.
\]
Put $\widehat{\Lambda(G)}_S$ for the $p$-adic completion of the localisation $\Lambda(G)_S$. \medskip

The extension $\mathbb{Q}(\mu_{p^{\infty}})$ of $\mathbb{Q}$ obtained by adjoining all $p$-power roots of 1 contains a unique extension of $\mathbb{Q}$ with Galois group isomorphic to $\mathbb{Z}_p$. We denote this extension by $\mathbb{Q}_{cyc}$, the cyclotomic $\mathbb{Z}_p$-extension of $\mathbb{Q}$. If $L$ is any number field, then the cyclotomic $\mathbb{Z}_p$-extension of $L$ is defined as $L_{cyc}:= L\mathbb{Q}_{cyc}$. For any number field $L$, the ring of integers of $L$ is denoted by $O_L$. \medskip

Throughout $F$ will denote a totally real number field of degree $r:=r_F := [F:\mathbb{Q}]$. Let $\Sigma := \Sigma_F$ denote a finite set of finite places of $F$. If $L$ is an extension of $F$, then we put $\Sigma_L$ for the set of places of $L$ above $\Sigma$. If there is no confusion we will often write $\Sigma$ for $\Sigma_L$. For any subset $O$ of $F$, we write $O^+$ for the set of totally positive elements of $O$. Throughout $F_{\infty}$ will denote a totally real Galois extension of $F$ such that 
\begin{enumerate}
\item $F_{cyc} \subset F_{\infty}$.
\item $F_{\infty}$ is unramified outside $\Sigma$. 
\item $G:= Gal(F_{\infty}/F)$ is a $p$-adic Lie group. 
\end{enumerate}

We put $A_F(G)$ (often written simply as $A(G)$, where $F$ is clear from the context) for the ring $\widehat{\Lambda(G)}_S[[q]]$ of all formal power series 
\[
a_0 + \sum_{\mu \in O_F^+} a_{\mu} q^{\mu}.
\]

\section{Introduction} The theory of $p$-adic modular forms essentially began with the paper of Serre \cite{serre:1973}. It was generalised by Katz \cite{Katz:1978} and Deligne-Ribet \cite{DeligneRibet:1980} and used to construct $p$-adic $L$-functions for $CM$ and totally real number fields respectively. The theory of $\Lambda$-adic modular forms was systematically developed by Hida. Since then they have formed a central tool in number theory and have most notably been used to prove main conjectures of commutative Iwasawa theory (Wiles \cite{Wiles:1990}, Skinner-Urban \cite{SkinnerUrban:2014} etc.). The main conjecture of non-commutative Iwasawa theory was formulated by Coates-Fukaya-Kato-Sujatha-Venjakob \cite{CFKSV:2005} for elliptic curves without complex multiplication and more generally in Fukaya-Kato \cite{FukayaKato:2006}. In an unpublished manuscript Kato \cite{Kato:2006} proved a case of non-commutative main conjecture for totally real fields by computing $K_1(\Lambda(G))$ and $K_1(\Lambda(G)_S)$ for a certain group $G$ and then proving congruences between certain between abelian $p$-adic zeta functions by proving the congruences first between $\Lambda$-adic Hilbert Eisenstein series. Abelian $p$-adic zeta functions appear in constant terms of these Eisenstein series (see theorem \ref{DR}). At the end of the paper Kato mentions the following question of Fukaya - Is there a $\Lambda$-adic modular form, with non-commutative ring $\Lambda$, whose constant term is the non-commutative $p$-adic $L$-function. We cannot answer this question completely but we do construct a $q$-expansion (in certain cases; see theorem \ref{theoremmain} for a precise statement) whose constant term is a non-commutative $p$-adic zeta function. The evaluation of this $q$-expansion at Artin characters is closely related to Hilbert Eisenstein series (see corollary \ref{corollaryevaluation}). 

The content of the article are as follows: in section \ref{sectionlambdaadicseries} we recall the result of Deligne and Ribet on Hilbert Eisenstein series. In section \ref{sectionmwcongruences} we prove the M\"{o}bius-Wall congruences for the Eisenstein series from section \ref{sectionlambdaadicseries}. As well as giving a slight generalisation of the congruences proven by Ritter-Weiss \cite{RitterWeiss:2010} this section simplifies the exposition. As usual the congruences are actually proven directly for non-constant coefficients of the standard $q$-expansion of the Eisenstein series in theorem \ref{DR}. The congruence for the constant terms, i.e. $p$-adic zeta functions,  can then be deduced from the $q$-expansion principal for Hilbert modular forms. These congruences are used in \cite{RitterWeiss:2010}, \cite{Kakde:2013} (generalising \cite{Kato:2006}) to construct non-commutative $p$-adic zeta function and prove the main conjecture for totally real number fields. In any case, we get the M\"{o}bius-Wall congruences for the $\Lambda$-adic Eisenstein series in theorem \ref{DR}. In section \ref{sectionk1} we give a description of $K_1(A_{\mathbb{Q}}(G))$ for certain $G$ (see \ref{theoremcongruences} for details). For simplicity we work only over $\mathbb{Q}$ but the result should hold over other totally real number fields. In section \ref{sectionqexpansion} we use this description along with the M\"{o}bius-Wall congruences for $\Lambda$-adic Eisenstein series to construct an element in $K_1(A_{\mathbb{Q}}(G))$ whose constant term equals the non-commutative $p$-adic zeta function.

\section{$\Lambda$-adic modular Eisenstein series} \label{sectionlambdaadicseries} In this section we assume that $G$ is commutative i.e. $F_{\infty}/F$ is an abelian extension. Recall the following result of Deligne and Ribet. 

\begin{theorem}[Deligne-Ribet \cite{DeligneRibet:1980}, theorem 6.1] \label{DR} There exists a $\Lambda(G)$-adic $F$ Hilbert modular Eisenstein series $\mathcal{E}(F_{\infty}/F)$ with standard $q$-expansion given by 
\[
2^{-r}\zeta(F_{\infty}/F) + \sum_{ \mu \in O_F^+} \left( \sum_{\mu \in \mathfrak{a}} \frac{\sigma_{\mathfrak{a}}}{N_F\mathfrak{a}} \right) q^{\mu},
\]
where $\zeta(F_{\infty}/F)$ is the $p$-adic zeta function, $\mathfrak{a}$ runs through all  ideals of $O_F$ coprime to $\Sigma$, $\sigma_{\mathfrak{a}} \in G$ is the Artin symbol of $\mathfrak{a}$, $N_F\mathfrak{a} \in \mathbb{Z}_p$ is the norm of $\mathfrak{a}$ and $q^{\mu} = e^{2\pi i tr_{F/\mathbb{Q}}(\mu)}$.

In particular, for any finite order character $\chi$ of $G$ and any positive integer $k$ divisible by $p-1$, the evaluation of $\mathcal{E}(F_{\infty}/F)$ at $\chi \kappa^k$ (here $\kappa$ is the cyclotomic character of $F$) has standard $q$-expansion
\[
2^{-r}L_{\Sigma}(\chi, 1-k) + \sum_{\mu \in O_F^+} \left(  \sum_{\mu \in \mathfrak{a}}\chi(\sigma_{\mathfrak{a}})N_F\mathfrak{a}^{k-1} \right) q^{\mu}. 
\]

\end{theorem}

\begin{proposition} If $\beta \in O_F^+$ divisible only by primes in $\Sigma$, then there exists a Hecke operator $U_{\beta}$ such that the action of $U_{\beta}$ on the standard $q$-expansion of $\Lambda(G)$-adic forms is as follows: if the standard $q$-expansion of $f$ is 
\[
c_0 + \sum_{\mu \in O_F^+} c(\mu) q^{\mu},
\]
then the standard $q$-expansion of $f|_{U_{\beta}}$ is 
\[
c_0 + \sum_{\mu \in O_F^+} c(\beta\mu) q^{\mu}.
\]
\end{proposition}

\begin{proof} See \cite[lemma 6]{RitterWeiss:congruences}. 
\end{proof}

Let $K$ be a subfield of $F$. Then the Hilbert modular variety of $K$ can be diagonally embedded in that of $F$. Restricting Hilbert modular forms on $F$ along this diagonal gives Hilbert modular forms over $K$. We denote this map by $Res_{F/K}$. 

\begin{proposition} If the standard $q$-expansion of $f$ is $c_0+ \sum_{\mu \in O_F^+} c(\mu)q^{\mu}$, then the standard $q$-expansion of $Res_{F/K}(f)$ is 
\[
c_0 + \sum_{\eta \in O_K^+} \left( \sum_{ \mu : tr_{F/K}(\mu) = \eta} c(\mu) \right) q^{\eta},
\]

\end{proposition}

\section{The M\"{o}bius-Wall congruences for the Eisenstein series} \label{sectionmwcongruences} In this section we assume that $G$ is a $p$-adic Lie group. 
Let 
\[
S^o(G) := \{U : U \text{ is an open subgroup of } G\}
\]
Put $F_U := F_{\infty}^U$, the field fixed by $U$ and put $K_U := F_{\infty}^{[U,U]}$, the field fixed by the commutator subgroup of $U$. Therefore $Gal(K_U/F_U) = U^{ab}$, the abelianisation of $U$. For $V, U \in S^o(G)$, with $V \subset U$, the transfer homomorphism $ver: U^{ab} \rightarrow V^{ab}$ induces a ring homomorphism 
\[
ver : A(U^{ab}) \rightarrow A(V^{ab}), 
\]
which is identity on the coefficients and $q$. If $V$ is a normal subgroup of $U$, then we can define a map 
\[
\sigma^U_V : A(V^{ab}) \rightarrow A(V^{ab}),
\]
given by 
\[
x \mapsto \sum_{g \in U/V} gxg^{-1}.
\]
Recall the definition of M\"{o}bius function on finite groups. It takes value 1 on the trivial group and then defined recursively as 
\[
\sum_{P' \subset P} \mu(P') = 0.
\]

\begin{theorem} \label{MWC} For every $V, U \in S^o(G)$, with $V$ a normal subgroup of $U$ we put
\[
\mathcal{G} := \sum_{V \subset W \subset U} \mu(W/V) ver\left(Res_{F_W/F_U}\left(\mathcal{E}(K_W/F_W)\right)|_{U_{[U:W]}}\right).
\]
Then the standard $q$-expansion of $\mathcal{G}$ lies in $Im(\sigma^U_V)$.
\end{theorem}

\begin{proof} We follow the proof in \cite[lemma 6]{RitterWeiss:2010}. Let $\mu \in O_{F_U}^+$. Then the $\mu$th coefficient of the standard $q$-expansion of $\mathcal{G}$ is 
\[
\sum_{V \subset W \subset U} \mu(W/V) ver \left( \sum_{(\alpha, \mathfrak{a})} \frac{\sigma_{\mathfrak{a}}}{N_{F_W}\mathfrak{a}} \right),
\]
where $\alpha$ in the second summation runs through all element of $O_{F_W}^+$ such that $tr_{F_W/F_U}(\alpha) = [U:W]\mu$ and $\mathfrak{a}$ runs through integral ideals of $F_W$ coprime to $\Sigma$ and containing $\alpha$. Take $M$ to be the set of all pairs $(\alpha, \mathfrak{a})$ with $\alpha \in O_{F_V}^+$ such that $tr_{F_V/F_U}(\alpha) = [U:V]\mu$ and $\mathfrak{a}$ an integral ideal of $F_V$ coprime to $\Sigma$ and containing $\alpha$. Then $U$ acts on $M$ and the above sum can be written as 
\[
\sum_{V  \subset W \subset U} \mu(W/V) \left( \sum_{(\alpha, \mathfrak{a}) \in M^W} \frac{ \sigma_{\mathfrak{a}}}{(N_{F_V}\mathfrak{a})^{1/[W:V]}}\right).
\]
Now fix $(\alpha, \mathfrak{a}) \in M$ and let $W_0$ be the stabiliser of $(\alpha, \mathfrak{a})$. Then the coefficient of $\sigma_{\mathfrak{a}}$ in the above sum is
\[
\sum_{V \subset W \subset W_0} \mu(W/V) \frac{1}{(N_{F_V}\mathfrak{a})^{1/[W:V]}} = \sum_{V \subset W \subset W_0} \mu(W/V) (N_{F_{W_0}}\mathfrak{a})^{-[W_0:W]}.
\]
We get the coefficient of $\sigma_{g(\mathfrak{a})}$ for every $g \in U/W_0$. Hence to show the congruence it suffices to show that for any finite group $P$ and any unit $r$ in $\mathbb{Z}_p$ we have
\begin{equation} \label{mobiuscongruence}
\sum_{P' \subset P} \mu(P')r^{[P:P']} \equiv 0 (\text{mod } |P|\mathbb{Z}_p).
\end{equation}
We use \cite[corollary 3.9]{HawkesIsaacsOzaydin:1989}. Let $|P| = p^k\cdot t$ with $t$ an integer co-prime to $p$. Let $t'$ be a divisor of $t$. By taking $n = p^k\cdot t'$, the subgroup $H$ to be the identity we deduce from \emph{loc. cit.} that $\sum_{|P'|} \mu(P')$ is divisible by $p^k$, where $P'$ runs through all subgroups of $P$ whose order divides $p^k \cdot t'$. Since this holds for arbitrary $t'$ we deduce that the sum $\sum_{P'} \mu(P')$ is divisible by $p^k$, where $P'$ runs through all subgroups of $P$ whose order is divisible by $t'$ and divides $p^k \cdot t'$. Now by \cite[corollary 4.9]{HawkesIsaacsOzaydin:1989} we have that $p^{k'}$ divides $p \cdot \mu(P')$, where $p^{k'}$ is the largest power of $p$ dividing $|P'|$. Therefore $\mu(P') r^{[P:P']} \equiv \mu(P')z^{t'} (\text{mod } |P| \mathbb{Z}_p)$ for any subgroup $P'$ of $P$ or order $p^{k'} \cdot t'$ and where $z$ is the $(p-1)$st root of $1$ in $\mathbb{Z}_p$ congruence to $r$ modulo $p$. Therefore
\[
\sum_{P'} \mu(P') r^{[P:P']} \equiv z^{t'}(\sum_{P'} \mu(P')) \equiv 0 (\text{mod } |P|\mathbb{Z}_p),
\]
where the $P'$ runs through all subgroups $P'$ of $P$ whose is divisible by $t'$ and divides $p^k \cdot t'$. This proves congruence in equation (\ref{mobiuscongruence}) and hence the theorem.

\end{proof}

\begin{remark} We may replace $\mathcal{G}$ by $Res_{F_U/F}(\mathcal{G})$ and the conclusion of the theorem still clearly holds. Though this is not important here, in cases of Eisenstein series over other groups (i.e. other than $GL_2$) this may be useful because there are cases when $q$-expansion principal may be known to hold over $F$ but not over extensions of $F$. 
\end{remark}

\section{$K_1$ of some Iwasawa algebras} \label{sectionk1} Detailed proofs of results in this section will appear in \cite{BurnsKakde:algebraic}. From now on we also assume that $F = \mathbb{Q}$. Let $G$ be a compact $p$-adic Lie group of the form $H \rtimes \Gamma$, where $H \cong \mathbb{Z}_p^d$ and $\Gamma$ is an open subgroup of $\mathbb{Z}_p^{\times}$ and containing $1+p\mathbb{Z}_p$. Furthermore, we assume that the action of $\Gamma$ on $H$ is diagonal. Put $\Gamma_0 := \Gamma$ and $\Gamma_i := 1+p^i \mathbb{Z}_p \subset \mathbb{Z}_p^{\times}$ for $i \geq 1$. We put $\delta := [\Gamma: \Gamma_1]$. Put $G_i := H \rtimes \Gamma_i$ for $i \geq 0$. Let $\mathcal{A}(G)$ be the free abelian group generated by absolutely irreducible finite order (Artin) representations of $G$. Then we have a natural map 
\[
Det: K_1(A(G)) \rightarrow Hom(\mathcal{A}(G), A(\Gamma)^{\times}).
\]
Define $SK_1(A(G)):= Ker(Det)$. Put $K_1'(A(G)) := K_1(A(G))/SK_1(A(G))$. 

\begin{remark} We expect $SK_1(A(G))$ to be trivial in this case but make no attempt to prove it here. This would be in analogy with the fact that $SK_1(\Lambda(G))$ is trivial (\cite[proposition 12.7]{Oliver:1988}). 
\end{remark}

\begin{definition} Define a map 
\[
\theta := \prod_{i \geq 0} \theta_i : K_1'(A(G)) \rightarrow \prod_{i \geq 0} A(G_i^{ab})^{\times},
\]
where each $\theta_i$ is the composition $K_1'(A(G)) \rightarrow K_1'(A(G_i)) \rightarrow A(G_i^{ab})^{\times}$ of the norm map and the natural surjection. 
\end{definition}

\noindent{\bf Some more maps:} Let $ 0 \leq j \leq i$. We have two natural maps
\[
N:= N_{i,j} : A(G_j^{ab})^{\times} \rightarrow A(G_i/[G_j, G_j])^{\times}
\]
and the natural projection
\[
\pi := \pi_{i,j} : A(G_i^{ab}) \rightarrow A(G_i/[G_j, G_j]).
\]
We have the transfer homomorphism $ver: G_j^{ab} \rightarrow G_i^{ab}$ which induces a ring homomorphism, again denoted by $ver$
\[
ver: A(G_j^{ab}) \rightarrow A(G_i^{ab}),
\]
which acts as identity on $q$. We have a $\mathbb{Z}_p$-linear map from section \ref{sectionmwcongruences}
\[
\sigma_i := \sigma^G_{G_i} : A(G_i^{ab}) \rightarrow A(G_i^{ab}),
\]
The image of the map $\sigma_i$ lies in the subring $A(G_i^{ab})^{G}$, the part fixed by $G$. In fact, the image $\sigma_i$ is an ideal in this ring (but not in the ring $A(G_i^{ab})$). 

\begin{definition} Let $\tilde{\Phi} \subset \prod_{i \geq 0} \left(A(G_i^{ab})^{\times}\right)^G$ consisting of all tuples $(x_i)_i$ satisfying the congruence
\begin{itemize}
\item[(C)] $ver(x_{i-1}) \equiv x_i (\text{mod } Im(\sigma_i))$
\end{itemize}
\end{definition}

\begin{definition} Let $\Phi \subset \tilde{\Phi}$ consisting of all tuples $(x_i)_i$ satisfying the functoriality 
\begin{itemize}
\item[(F)] For all $0 \leq j \leq i$
\[
N_{i,j}(x_j) = \pi_{i,j}(x_i).
\] 
\end{itemize} 
\end{definition}

We define one more map before stating the main theorem of this section. We define $\eta_0 : A(G_0^{ab}) \rightarrow A(G_0^{ab})$ to be 
\[
\eta_0(x) = \frac{x^{\delta}}{\prod_{k=0}^{\delta -1} \tilde{\omega}^k(x)},
\]
where $\omega$ is a character of $G^{ab}$ inflated from an order $\delta$ character of $\Gamma$ and $\tilde{\omega}$ is the map induced by $g \mapsto \omega(g)g$. For every $i \geq 1$, we define 
\[
\eta_i : A(G_i^{ab}) \rightarrow A(G_i^{ab})
\]
by 
\[
x \mapsto \frac{x^p}{\prod_{k=0}^{p-1} \tilde{\omega}_i^k(x)},
\]
where $\omega_i$ is a character of $G_i^{ab}$ inflated from an order $p$ character of $\Gamma_i$ and $\tilde{\omega}_i$ is the map on $A(G_i^{ab})$ induced by $g \mapsto \omega_i(g)g$. We put 
\[
\eta := \prod_{i \geq 0} \eta_i : \prod_{i \geq 0} A(G_i^{ab}) \rightarrow A(G_i^{ab})
\]

\begin{theorem} \label{theoremcongruences} 
\begin{enumerate}
\item $\theta$ induces an isomorphism between $K_1'(A(G))$ and $\Phi$. 
\item The inclusion $\Phi \hookrightarrow \tilde{\Phi}$ has a natural section.
\end{enumerate}

\end{theorem}

\begin{proof} We give a rough idea of the proof with details and more general results appearing in \cite{BurnsKakde:algebraic}. \\
(1) By definition of $K_1'$  and the fact that every irreducible Artin representation of $G$ is induced from a one dimensional Artin representation of $G_i$ for some $i$ (\cite[proposition 25]{Serre:representationtheory}) it is clear that the map $\theta$ is injective.  It is also easy to show that the image of $\theta$ lies in $\Phi$. Let $(x_i)_i \in \Phi$. Then the inverse image of $(x_i)_i$ in $K_1'(A(G))$ is constructed as follows: firstly we may and do assume that $x_0=1$ since the map $K_1(A(G)) \rightarrow A(G_0^{ab})^{\times}$ is surjective. For every $i \geq 1$ put $y_i := \frac{x_i}{ver(x_{i-1})}$. Define 
\[
X := \prod_{i \geq 1} \eta_i(y_i)^{\frac{1}{(p-1)p^i}} \in K_1'(A(G)).
\]
There are several points that need explaining - firstly, $\eta_i(y_i)$ lies in $A(G_i^{ab})^{\times}$. However, it can be show that the map $K_1'(A(G_i)) \rightarrow A(G_i^{ab})^{\times} \xrightarrow{\eta_i} \eta_i(A(G_i^{ab})^{\times})$ splits and hence $\eta_i(y_i)$ makes sense as an element of $K_1'(A(G))$. It can be shown using conditions (F) and (C) that $\eta_i(y_i)$ is a $(p-1)p^i = p[G:G_i]$th power in $K_1'(A(G))$. Using the fact that $x_0=1$, we can show that $\eta_i(y_i)$ actually lies in the image, denoted by $\overline{K}_1(A(G),J)$ of $K_1(A(G),J)$ in $K_1(A(G))$ and that it has a unique $(p-1)p^i$th root in $\overline{K}_1(A(G),J)$. Here $J$ is the kernel of $A(G) \rightarrow A(G^{ab})$. Lastly, one needs to show that the infinite product converges. Each of this step is non-trivial and crucially uses integral logarithm map (or rather it generalisation to rings like $A(G)$) of R. Oliver and M. Taylor.   

(2) By the above we may define the natural section as follows: Let $(x_i)_i \in \tilde{\Phi}$ and we may again assume that $x_0=1$. Define $z_0=1$ and for $i \geq 1$ define
\[
z_i = \prod_{j \geq i} \eta_i(y_i)^{\frac{1}{p^{j-i+1}}} \in A(G_i^{ab})^{\times},
\]
with $y_i$ defined as above. Then one can check that $(z_i)_i \in \Phi$ and gives a natural section of the inclusion. 
\end{proof}

\begin{definition} We define the natural section of the inclusion $\Phi \subset \tilde{\Phi}$ given by the above theorem by $s$. 
\end{definition}

\begin{corollary} \label{corophitilde} If $(x_i)_i \in \tilde{\Phi}$, then there is a unique element $x \in K_1'(A(G))$ such that $\theta(x) = s((x_i)_i)$. 
\end{corollary}

\section{Non-commutative $q$-expansions} \label{sectionqexpansion} We continue with the notation of the previous section. Therefore $F = \mathbb{Q}$. Let $F_i := F_{\infty}^{G_i}$ and $K_i := F_{\infty}^{[G_i, G_i]}$. We put $\mathcal{E}_i$ for the standard $q$-expansion of the $\Lambda(G_i^{ab})$-adic Hilbert modular form $Res_{F_i/\mathbb{Q}}(\mathcal{E}(K_i/F_i))|_{U_{\delta p^i}}$ from section \ref{sectionlambdaadicseries}. As $F_1$ is an abelian extension of $\mathbb{Q}$ and $G_1$ is pro-$p$ we know by the theorem of Ferrero-Washington \cite{FerreroWashington:1979} that $\mathcal{E}_i \in A(G_i^{ab})^{\times}$ (it is enough to show that the constant term, i.e. the $p$-adic zeta functions $\zeta(K_i/F_i)$, of $\mathcal{E}_i$ are units in $\Lambda(G_i^{ab})_S$. This is well-known, for example see \cite[lemma 1.7 and 1.14]{Kakde:2011}).

\begin{theorem} \label{theoremmain} The tuple $(\mathcal{E}_i)_i$ lies in $\tilde{\Phi}$. Hence there exists $\mathcal{E} \in K_1'(A(G))$ such that $\theta(\mathcal{E}) = s((\mathcal{E}_i)_i)$. The ``constant term" of $\mathcal{E}$, i.e. its image under the map $K_1'(A(G)) \rightarrow K_1'(\widehat{\Lambda(G)}_S)$ mapping $q$ to 0, takes $\mathcal{E}$ to the $p$-adic zeta function for $F_{\infty}/F$. 
\end{theorem}

\begin{proof} It is clear from the explicit expression that each $\mathcal{E}_i$ is fixed under the conjugation action of $G$. The congruence condition (C) follows from the M\"{o}bius-Wall congruences as follows: taking $V= G_i$ and $U=G_1$ and noting that $\mu(\mathbb{Z}/p^n \mathbb{Z})$ is zero unless $0 \leq n \leq 1$, we get that the standard $q$-expansion of 
\begin{equation}
\label{equationres}
ver(Res_{F_{i-1}/F_1}(\mathcal{E}(K_{i-1}/F_{i-1}))|_{U_{p^{i-1}}}) - Res_{F_i/F_1}(\mathcal{E}(K_i/F_i))|_{U_{p^i}}
\end{equation}
lies in $Im(\sigma^{G_1}_{G_i})$ by theorem \ref{MWC}. Now $ver(\mathcal{E}_{i-1}) - \mathcal{E}_i$ is obtained by applying $Res_{F_1/\mathbb{Q}}$ and $U_{\delta}$ to (\ref{equationres}) and taking its standard $q$-expansion. Hence $ver(\mathcal{E}_{i-1}) - \mathcal{E}_i \in Im(\sigma^{G_1}_{G_i})$. Next we use that $\delta$ is an integer and a unit in $\mathbb{Z}_p$, and $ver(\mathcal{E}_{i-1}) - \mathcal{E}_i$ is fixed by the action of $G$. Hence $ver(\mathcal{E}_{i-1}) - \mathcal{E}_i \in Im(\sigma_i)$. Hence $(\mathcal{E}_i)_i$ lies in $\tilde{\Phi}$. The second assertion follows from corollary \ref{corophitilde}.

The last assertion follows from the commutative diagram 
\[
\xymatrix{
K_1'(A(G)) \ar[r] \ar[d]_{\theta} & K_1'(\widehat{\Lambda(G)}_S) \ar[d]^{\theta} \\
\prod_{i \geq 0} A(G_i^{ab})^{\times} \ar[r] & \prod_{i \geq 0} \widehat{\Lambda(G_i^{ab})}_S^{\times},
}
\]
where the horizontal maps are $q \mapsto 0$.
\end{proof}

The following corollary tells us something about evaluation of $\mathcal{E}$ at elements of $\mathcal{A}(G)$. 

\begin{corollary} \label{corollaryevaluation} Let $\rho \in \mathcal{A}(G)$. Then there exist $i$ and a one dimensional character $\chi$ of $G_i$ such that $\rho = Ind_{G_i}^G \rho$. Then 
\[
\rho(\mathcal{E}) = \prod_{j \geq i} \left( \chi(\eta_j(\mathcal{E}_j))^{\frac{1}{[\Gamma_j:\Gamma_i]}} \right).
\]
\end{corollary}

\begin{remark} There are many examples of the totally real extensions of $\mathbb{Q}$ with Galois group $G$ whose form is as in the previous section. For example if $p$ is an irregular prime, then the maximal abelian pro-$p$ extension of $\mathbb{Q}(\mu_{p^{\infty}})^+$ (the maximal totally real subfield of $\mathbb{Q}(\mu_{p^{\infty}})$ is isomorphic to $\mathbb{Z}_p^d$ for some positive integer $d$. If Vandiver's conjecture is true for $p$ then the action of $Gal(\mathbb{Q}(\mu_{p^{\infty}})^+/\mathbb{Q})$ on $\mathbb{Z}_p^d$ is diagonal. 
\end{remark}

\begin{remark} Even though the modifications we have to make to the $\Lambda$-adic Eisenstein series to get a lift to the non-commutative case are somewhat complicated, the result is rather formal i.e. if a ``$q$-expansion" satisfies a congruence, then its modification given  above can be lifted. The congruences seem to hold for Eisenstein series over other groups (see for example \cite{Bouganis:constantterm}). Hence their modifications should also have a lift. Is there a more conceptual description of these lifts? 
\end{remark}

\begin{remark} We also remark that $(\mathcal{E}_i)_i$ does not lie in $\Phi$ because they do not satisfy the functoriality condition (F). 
\end{remark}

\bibliographystyle{plain}
\bibliography{mybib2}

\begin{thebibliography}{10}

\bibitem{Bouganis:constantterm}
{Bouganis, T.}
\newblock {Non-abelian $p$-adic $L$-functions and Eisenstein series of unitary
  groups; the constant term method}.
\newblock In preparation.

\bibitem{BurnsKakde:algebraic}
{Burns, D. and Kakde, M.}
\newblock {Congruences in non-commutative Iwasawa theory}.
\newblock In preparation.

\bibitem{CFKSV:2005}
{Coates, J. and Fukaya, T. and Kato, K. and R. Sujatha and Venjakob, O.}
\newblock The ${GL_2}$ main conjecture for elliptic curves without complex
  multiplication.
\newblock {\em Publ. Math. IHES}, (1):163--208, 2005.

\bibitem{DeligneRibet:1980}
{Deligne, P and Ribet, Kenneth A.}
\newblock Values of abelian {L}-functions at negative integers over totally
  real fields.
\newblock {\em Inventiones Math.}, 59:227--286, 1980.

\bibitem{FerreroWashington:1979}
{Ferrero, B. and Washington, L.C.}
\newblock The {I}wasawa invariant $\mu_p$ vanishes for abelian number fields.
\newblock {\em Ann. of Math.}, 109:377--395, 1979.

\bibitem{FukayaKato:2006}
{Fukaya, T and Kato, K}.
\newblock A formulation of conjectures on p-adic zeta functions in
  non-commutative {I}wasawa theory.
\newblock In N.~N. Uraltseva, editor, {\em Proceedings of the St. Petersburg
  Mathematical Society}, volume~12, pages 1--85, March 2006.

\bibitem{HawkesIsaacsOzaydin:1989}
{Hawkes, T. and Isaacs, I.M. and \"{O}Zaydin, M.}
\newblock {On the M\"{o}bius function of a finite group}.
\newblock {\em {Rocky Mountain Journal of Mathematics}}, 19(4):1003--1034,
  1989.

\bibitem{Kakde:2011}
{Kakde, M.}
\newblock {Proof of the main conjecture of noncommutatve Iwasawa theory for
  totally real number fields in certain cases}.
\newblock {\em J. Algebraic Geom.}, 20:631--683, 2011.

\bibitem{Kakde:2013}
{Kakde, Mahesh}.
\newblock {The main conjecture of {I}wasawa theory for totally real fields}.
\newblock {\em {Invent. Math.}}, 193(3):539--626, 2013.

\bibitem{Kato:2006}
{Kato, K.}
\newblock Iwasawa theory of totally real fields for {G}alois extensions of
  {H}eisenberg type.
\newblock Very preliminary version, 2006.

\bibitem{Katz:1978}
{Katz, N. M.}
\newblock p-adic ${L}$-functions for {CM} fields.
\newblock {\em Inventiones Math.}, 49:199--297, 1978.

\bibitem{Oliver:1988}
{Oliver, R.}
\newblock {\em Whitehead Groups of Finite Groups}.
\newblock Number 132 in London Mathematical Society Lecture Note Series.
  Cambridge University Press, 1988.

\bibitem{RitterWeiss:congruences}
{Ritter, J. and Weiss, A.}
\newblock Congruences between abelian pseudomeasures.
\newblock {\em Mathematical Research Letters}, 15(4):715--725, July 2008.

\bibitem{RitterWeiss:2010}
{Ritter, J. and Weiss, A.}
\newblock On the `main conjecture' of equivariant {I}wasawa theory.
\newblock {\em Journal of the AMS}, 24:1015--1050, 2011.

\bibitem{serre:1973}
{Serre, J-P.}
\newblock Formes modulaires et fonctions z{\^e}ta $p$-adiques.
\newblock In {\em Modular functions of one variable, III}, volume LNM 350,
  pages 191--268, 1973.

\bibitem{Serre:representationtheory}
{Serre, J-P}.
\newblock {\em Linear Representation of Finite Groups}.
\newblock Number~42 in {Graduate Text in Mathematics}. Springer-Verlag, 1977.

\bibitem{SkinnerUrban:2014}
{Skinner, Christopher and Urban, Eric}.
\newblock {The Iwasawa Main Conjecture for $GL_2$}.
\newblock {\em {Invent. Math.}}, 195:1--277, May 2014.

\bibitem{Wiles:1990}
{Wiles, A}.
\newblock The {I}wasawa conjecture for totally real fields.
\newblock {\em Ann. of Math.}, 131(3):493--540, 1990.

\end{thebibliography}

\end{document}